\documentclass{siamltex}%

\usepackage{amsfonts}
\usepackage{amsmath}%
\usepackage{amssymb}%
\usepackage{graphicx}

\usepackage[T1]{fontenc}
\usepackage[utf8]{inputenc}
\usepackage{mathtools} 
\usepackage{multirow}
\usepackage{lipsum}
\usepackage{algorithmic}
\usepackage[ruled,vlined]{algorithm2e}
\usepackage{cleveref}
\usepackage{cite}
\usepackage{epstopdf}
\usepackage{float}
\usepackage{diagbox}
\usepackage{color}
\usepackage[margin=1in]{geometry}
\usepackage{lscape}
\graphicspath{ {./png/} }
\usepackage{epsfig}
\usepackage{graphics}
\usepackage{bm}
\usepackage{latexsym}
\usepackage{cancel}
\usepackage{epic,eepic,euscript,verbatim,afterpage,stmaryrd,paralist}
\usepackage{amscd}
\usepackage{mathrsfs}
\usepackage{fancyhdr,fancybox}
\usepackage{listings}
\usepackage{url} 
\usepackage{enumerate,algorithm2e}
\usepackage{array}
\usepackage{ulem}
\usepackage{epstopdf,subfigure}
\usepackage{appendix}
\usepackage{comment}

\newcommand{\Dt}{\Delta}
\newcommand{\be}{\begin{equation}}
\newcommand{\ee}{\end{equation}}
\newcommand{\ba}{\begin{array}}
\newcommand{\ea}{\end{array}}
\newcommand{\bea}{\begin{eqnarray}}
\newcommand{\eea}{\end{eqnarray}}
\newcommand{\beas}{\begin{eqnarray*}}
\newcommand{\eeas}{\end{eqnarray*}}

\providecommand{\U}[1]{\protect\rule{.1in}{.1in}}

\newtheorem{remark}[theorem]{Remark}

\begin{document}

\title{On Adaptive grad-div parameter selection \thanks{The research was partially supported by NSF grant DMS-2110379.} }
\author{Xihui Xie\thanks{Department of Mathematics, University of Pittsburgh, Pittsburgh, PA 15260 (xix55@pitt.edu).} }
\maketitle

\begin{abstract}
    We propose, analyze and test a new adaptive penalty scheme that picks the penalty parameter $\epsilon$ element by element small where $\nabla\cdot u^h$ is large. We start by analyzing and testing the new scheme on the most simple but interesting setting, the Stokes problem. Finally, we extend and test the algorithm on the incompressible Navier Stokes equation on complex flow problems. Tests indicate that the new adaptive-$\epsilon$ penalty method algorithm predicts flow behavior accurately. The scheme is developed in the penalty method but also can be used to pick a grad-div stabilization parameter.
\end{abstract}

% \subjclass[2010]{65M12, 65M60}
% \AMclass{65M12, 65M60}
%===============================================

\section{Introduction}
%{\cc While this paper considers the Stokes problem, we start by introducing penalty method using Navier Stokes equation as an example.}
Consider the incompressible Navier-Stokes equations (NSE) with no-slip boundary condition:
\begin{equation}
\begin{split}
    &u_t+u\cdot\nabla u+\nabla p-\nu\Delta u=f,\ \text{and}\ \nabla\cdot u=0,\ \text{in}\ \Omega\times[0,T],\\
    &u=0,\ \text{on}\ \partial\Omega\times[0,T],\ \text{and}\ u(x,0)=u_0(x),\ \text{in}\ \Omega.
\end{split}
\end{equation}
Here $u$ is the velocity, $f$ is the known body force, $p$ is the pressure, and $\nu$ is the viscosity. \\
The penalty approximation to the Navier-Stokes equations replaces $\nabla\cdot u=0$ by $\nabla\cdot u+\epsilon p=0$ and eliminates the pressure. This uncouples velocity and pressure, and the resulting system is much easier to solve:
\begin{equation}\label{penaly}
    \begin{aligned}
        &u_{\epsilon,t}+u_\epsilon\cdot\nabla u_\epsilon+\frac{1}{2}(\nabla\cdot u_\epsilon)u_\epsilon-\nu\Dt u_\epsilon-\nabla(\frac{1}{\epsilon}\nabla\cdot u_\epsilon)=f\ \text{in}\ \Omega\times[0,T],\\
        &u_\epsilon=0, \text{on}\ \partial\Omega\times[0,T],\ \text{and}\ u_\epsilon(x,0)=u_0(x),\ \text{in}\ \Omega.
    \end{aligned}
\end{equation}
Here $u_\epsilon\cdot\nabla u_\epsilon+\frac{1}{2}(\nabla\cdot u_\epsilon)u_\epsilon$ is the modified bilinear term introduced by Temam \cite{teman1968}. This bilinear term ensures the dissipativity of the system \eqref{penaly}. 
Supposing the spacial discretization, a simple penalty method is given as follows. Given $u^n\approx u(x,t_n), k_n=t^{n+1}-t^n$ the $n^{th}$ time step
\begin{equation}\label{nse}
\begin{split}
    &\frac{u_\epsilon^{n+1}-u_\epsilon^n}{k_n}+u_\epsilon^n\cdot\nabla u_\epsilon^{n+1}+\frac{1}{2}(\nabla\cdot u_\epsilon^{n})u_\epsilon^{n+1}-\nu\Delta u_\epsilon^{n+1}-\nabla (\frac{1}{\epsilon_{n+1}}\nabla\cdot u_\epsilon^{n+1})=f(t^{n+1}),\ \text{in}\ \Omega,\\
    &u_\epsilon^{n+1}=0,\ \text{on}\ \partial\Omega\  \text{and}\ u_\epsilon^0=u_\epsilon(x,0)=u_0(x),\ \text{in}\ \Omega.
\end{split}
\end{equation}
The term $-\nabla(\epsilon^{-1}\nabla\cdot u)$ also arises in artificial compression method and with grad-div stabilization.
Penalty methods require less computing time and reduced storage but still face two unsolved problems:
\begin{enumerate}
    \item How to recover the pressure accurately, and
    \item How to pick an effective value of the grad-div coefficient $\epsilon$.
\end{enumerate}
Herein we present a self-adaptive algorithm answering question 2. 

There are many papers devoted to the parameter choice of grad-div term for both grad-div stabilization problem and penalty problem. 
%Olshanskii and Reusken \cite{grad-div-olshanskii} analyzes the effect of the added grad-div term and shows added term having a stabilizing impact for small $\nu$ value. 
Jenkins, John, Linke, and Rebholz \cite{grad-div-jenkins2014parameter} found that the grad-div parameter for Stokes problem depends on the used norm, the solution, the finite element space and type of mesh used.
Ainsworth, Allendes, Barrenechea and Rankin \cite{stabilized2013adaptive} introduced an approach to select stabilization parameters for the Stokes problem.

The velocity error of penalty methods is also sensitive to the choice of $\epsilon$, see Bercovier and Engelman \cite{BERCOVIER1979181}. 
Care must be taken when choosing $\epsilon$. If $\epsilon$ is too large, it will poorly model incompressible flow. 
Choosing $\epsilon$ too small will cause numerical conditioning problems, see Hughes, Liu and Brooks \cite{hughes1979finite}. In \cite{hughes1979finite}, the authors introduced a theory for determining the penalty parameter, which depends on Reynolds number Re and viscosity $\mu$. 
The optimal choice of the penalty parameter also varies according to the time discretization schemes and space discretization schemes used, see Shen \cite{Shen1995penalty}.
With so many dependencies, an automatic choice of $\epsilon$ naturally becomes a problem to consider.
%In Section 1.1, with the detailed mathematical analysis, we formulate a new way to select the penalty parameter $\epsilon$ automatically. \\

In Layton and McLaughlin \cite{layton2019doublyadaptive} self-adaptive $\epsilon$ selection in time (but not in space) algorithms were developed, analyzed and tested. The basic idea in \cite{layton2019doublyadaptive} is to monitor $\|\nabla\cdot u^n\|$ and pick $\epsilon=\epsilon(t_n)$ to make $\|\nabla\cdot u^n\|<$ Tolerance in \eqref{nse} in the computation of $u^{n+1}$. 

The natural question we answer herein is: 
can we let $\epsilon=\epsilon(x,t)$ and pick $\epsilon(x,t_n)$ pointwise or element by element small where $\nabla\cdot u^h$ is large to enforce in a realizable sense
\begin{equation}\label{tol}
    \int_\Omega |\nabla\cdot u^h|^2\ dx<\text{Tolerance}^2.
\end{equation}
This means $\epsilon$ is chosen small where $\nabla\cdot u^h$ is large (and large where small). As a result, the term $(\epsilon^{-1}\nabla\cdot u^h,\nabla\cdot v^h)$ becomes nonlinear. To our knowledge, this natural idea has not been considered. 
Picking $\epsilon$ pointwise and elementwise are two related ideas, but the resulting two algorithms are different; see \eqref{vform2} and \eqref{vform} below.

The idea we use is the path of many adaptive methods: monitor the residual (the left-hand side of \eqref{tol}), localize the global tolerance \eqref{tol} and where the local residual $\int_\Dt |\nabla\cdot u^h|^2\ dx$ is large, pick $\epsilon_\Dt$ small (and visa versa). 
Picking $\epsilon$ locally in space leads to a nonlinear grad-div term in \eqref{nse}
quite amenable to numerical analysis. In the next sections, we start the detailed analysis and test of this idea using the simplest setting, the Stokes problem.

\subsection{Previous Work}
%Bernardi, Girault and Hecht \cite{hecht2003choix,hecht2003posteriori}  derived posterior error estimates for the Stokes problem with penalty and tested using local penalty parameters. 
Bernardi, Girault and Hecht \cite{hecht2003choix,hecht2003posteriori}  derived posterior error estimates for the Stokes problem with penalty. They performed the tests on adaptive meshes and also tested using local penalty parameters.
%{\cc something nice about Falk and Bernardi}
Falk \cite{falk1976finite} derived a new finite element method that uses the trial function, which is not div-free. By eliminating the constraint, one can use a simple finite elements, which inspired the proof in Section 3.2. 
Heavner and Rebholz \cite{tol2017locally} considered a local choice of grad-div stabilization parameter. And in numerical tests, they showed that local choice of stabilization parameter provides more accurate solutions.
\subsection{Formulation}
We begin the analysis and testing of this idea for the simplest interesting setting, the Stokes problem
\begin{equation}\label{stokes}
    -\nu\Dt u+\nabla p=f(x),\quad\quad \nabla\cdot u=0.
\end{equation}
On a bounded, open polygonal domain $\Omega$ subject to no-slip boundary conditions $u=0$ on $\partial \Omega$. Let $d$ denote the dimension of $\Omega$, $d=dim(\Omega)=2$ or $3$.

The penalty method replaces $\nabla\cdot u=0$ by $\nabla\cdot u_\epsilon+\epsilon p=0$ and eliminate pressure using $p=-\epsilon^{-1}\nabla\cdot u_\epsilon$:
\begin{equation}\label{eq:stokes-penalty}
    -\nu\Dt u_\epsilon-\nabla\left(\frac{1}{\epsilon}\nabla\cdot u_\epsilon\right)=f(x)\ \text{in}\ \Omega.
\end{equation}
Let $X^h\subset X:=(H^{0,1}(\Omega))^d, d=2$ or 3 denote a finite element space for the fluid velocity. $(\cdot,\cdot)$ is the $L^2$ inner product with norm $\|\cdot\|$ and $\Dt$ denotes a mesh element (so that $\int_\Omega \phi\ dx=\sum_\Dt \int_\Dt \phi\ dx$). The area/volume of a region D is denoted $|D|$. The $L^2(\Dt)$ norm on a mesh element $(\int_\Dt\phi^2\ dx)^{1/2}$ is denoted as $\|\phi\|_\Dt$.

The penalty approximation we consider to \eqref{stokes} is: find $u^h\in X^h$ such that
\begin{equation}\label{eq:tvform}
    \nu(\nabla u_\epsilon^h,\nabla v^h)+\sum_{\Dt}\int_\Dt \epsilon_\Dt^{-1}\nabla\cdot u_\epsilon^h\nabla\cdot v^h\ dx=(f, v^h),  \quad\forall v^h \in V^h.
\end{equation}
The idea is the same as behind most adaptive algorithms: Monitor the residual to control the error; localize a global residual tolerance; where the local residual $\|\nabla\cdot u^h\|_\Dt^2$ is large pick $\epsilon_\Dt$ small.

To develop this, we begin with the basic stability estimate. Setting $v^h=u^h$ in \eqref{eq:tvform} we find
\begin{equation*}
    \nu\|\nabla u_\epsilon^h\|^2+\sum_\Dt\int_\Dt\epsilon_\Dt^{-1}|\nabla\cdot u_\epsilon^h|^2\ dx=(f,u_\epsilon^h)=(f,u)+o(1),
\end{equation*}
\begin{equation*}
    \text{thus}\ \sum_\Dt \int_\Dt\epsilon_\Dt^{-1} |\nabla\cdot u_\epsilon^h|^2\ dx=\mathcal{O}(1).
\end{equation*}
This sugggests that globally halving (doubling) $\epsilon$ halves (doubles) $\|\nabla\cdot u_\epsilon^h\|^2$.

Next, we localize the global tolerance TOL %, or equivalently,
%so globally halving (doubling) $\epsilon$ at least halves (doubles) $\|\nabla\cdot u^h\|^2$\\
for $\|\nabla\cdot u_\epsilon^h\|$ as follows:\\
We seek $\|\nabla\cdot u_\epsilon^h\|^2\approx\frac{1}{2}$TOL$^2$ or
\begin{align*}
    \|\nabla\cdot u_\epsilon^h\|^2=\sum_\Dt\int_\Dt |\nabla\cdot u_\epsilon^h|^2\ dx&\approx \frac{1}{2}\text{TOL}^2 = \frac{1}{2}\sum_\Dt \frac{TOL^2}{|\Omega|}|\Dt|.
\end{align*}
Thus we define the local tolerance
\begin{equation*}
    \text{LocTol}_\Dt:=\frac{1}{2}\frac{TOL^2}{|\Omega|}|\Dt|,
\end{equation*}
and seek to enforce
\begin{equation*}
    \|\nabla\cdot u_\epsilon^h\|_\Dt^2\approx\text{LocTol}_\Dt.
\end{equation*}
If this local tolerance is satisfied, the global tolerance is satisfied:
\begin{equation*}
    \|\nabla\cdot u_\epsilon^h\|^2=\sum_\Dt\int_\Dt|\nabla\cdot u_\epsilon^h|^2\ dx\approx \sum_\Dt LocTol_\Dt=\frac{1}{2}TOL^2.
\end{equation*}

The usual procedure would be to select (on each triangle $\Dt$)  $\epsilon_{old}$, solve for $u_\epsilon^h$, compute the ratio
\begin{equation*}
    r=\frac{\text{LocTol}_\Dt}{\|\nabla\cdot u_\epsilon^h\|_\Dt^2 },
\end{equation*}
then adjust $\epsilon$ by $\epsilon_{new}=r\times \epsilon_{old}$ and resolve.
The first step is therefore (starting with $\epsilon_\Dt\equiv1$)
\begin{equation*}
    \epsilon_\Dt =\|\nabla\cdot u_\epsilon^h\|_\Dt ^{-2}\times\text{LocTol}_\Dt,
\end{equation*}
%using this value of $\epsilon$ gives the nonlinear grad-div term in variational forms and local tolerance $LocTol_\Dt\equiv\frac{1}{2}TOL^2\frac{|\Dt|}{|\Omega|}$. 
There are two options. Both result in a nonlinear discretization.\\
Option 1. Elementwise Penalty (EP)
\begin{equation*}
    \epsilon_\Dt:=\frac{LocTol_\Dt}{\|\nabla\cdot u^h_\epsilon\|_\Dt^2},
\end{equation*}
so that
\begin{align}\label{grad-div-vform2}
 \sum_\Dt \int_\Dt\epsilon_\Dt^{-1}\nabla\cdot u_\epsilon^h\nabla\cdot v^h\ dx&=\sum_\Dt LocTol_\Dt^{-1}\|\nabla\cdot u_\epsilon^h\|_\Dt^2\int_\Dt\nabla\cdot u_\epsilon^h\nabla\cdot v^h\ dx .
     %\text{or equivalently,}\quad \epsilon_\Dt :&=\frac{\text{LocalTol}_\Dt}{\|\nabla\cdot u_\epsilon^h\|_\Dt ^2},\\
     %\text{where LocalTol}_\Dt&\equiv\frac{1}{2}\frac{TOL^2}{|\Omega|}|\Dt|.
\end{align}
Then \eqref{eq:tvform} becomes: find $u^h_\epsilon\in X^h$ such that
\begin{equation}\label{vform2}
    \int_\Omega\nu\nabla u^h_\epsilon:\nabla v^h\ dx+\sum_\Dt\frac{1}{LocTol_\Dt}\|\nabla\cdot u^h_\epsilon\|^2_\Dt\int_\Dt\nabla\cdot u^h_\epsilon\nabla\cdot v^h\ dx=\int_\Omega f\cdot v^h\ dx.
\end{equation}
%In this paper, we consider the nonlinear grad-div term in the form of \eqref{grad-div-vform1},
%then \eqref{eq:tvform} becomes: find $u^h\in X^h$ such that
Option 2. Pointwise Penalty (PP)
\begin{equation*}
    \epsilon_\Dt(x):=\frac{LocTol_\Dt}{|\nabla\cdot u^h_\epsilon(x)|^2},
\end{equation*}
so that
\begin{align}\label{grad-div-vform1}
     \sum_\Dt \int_\Dt\epsilon_\Dt^{-1}\nabla\cdot u_\epsilon^h\nabla\cdot v^h\ dx&=\sum_\Dt LocTol_\Dt^{-1}\int_\Dt |\nabla\cdot u_\epsilon^h|^2\nabla\cdot u_\epsilon^h\nabla\cdot v^h\ dx .
     %\text{where LocalTol}_\Dt&\equiv\frac{1}{2}\frac{TOL^2}{|\Omega|}|\Dt|, 
\end{align}
Then \eqref{eq:tvform} becomes: find $u^h_\epsilon\in X^h$ such that
\begin{equation}\label{vform}
    \int_\Omega\nu\nabla u_\epsilon^h:\nabla v^h\ dx+\sum_\Dt\frac{1}{LocTol_\Dt}\int_\Dt |\nabla\cdot u_\epsilon^h|^2\nabla\cdot u_\epsilon^h\nabla\cdot v^h\ dx=\int_\Omega f\cdot v^h\ dx.
\end{equation}
%The result \eqref{vform} is a nonlinear discretization of a simple, linear Stokes problem.

We focus herein on the analysis of option 2 (PP) and numerical result of option 1 (EP). In option 2 (PP), the resulting nonlinearity is both strongly monotone and locally Lipschitz continuous, sharing structures with the p-Laplacian. Then, there is a well-trodden analytical path to be adapted here. Before proceeding, we address two points:
\begin{enumerate}
    \item Imposing the global condition locally suggests but does not imply the local condition is satisfied. This will be tested in our experiments Section 6.1. We adapt based on the local condition but aim for global TOL to be satisfied.
    \item No analysis herein addresses how to pick TOL. %to give a good approximation. This choice of tolerance can be found in paper %Case, Ervin, Linke and Rebholz \cite{rebholz-graddiv2011} 
    %Jenkins, John, Linke and Rebholz \cite{grad-div-jenkins2014parameter} and Heavner \cite{tol2017locally} addressing added grad-div term and penalty methods. 
    TOL is user supplied.
    %This is addressed in {\cc cite paper here} for an added grad-div term and in {\cc cite paper here} for penalty methods.
\end{enumerate}

Section 2 introduces some notation and preliminaries. Section 3 analyzes the stability and error for the Stokes problem of the new pointwise penalty (PP) method. In Section 4, algorithmic aspects are discussed for the Stokes problem and the Navier Stokes problem using the elementwise penalty (EP) method. Section 5, we present three numerical tests using the elementwise penalty (EP). The first two are for the Stokes problem and the third one is an extension to the Navier Stokes equations. Finally, in Section 6, we draw conclusions and point out future research directions.

\section{Notation and Preliminaries}
Let $H_0^1(\Omega)=\{u\in L^2(\Omega):\nabla u\in L^2(\Omega)\ \text{and}\ u|_{\partial\Omega}=0\ \text{in}\ L^2(\partial \Omega)\}$.
Let $X$ be the velocity space $=(H^{1}_0(\Omega))^d$, $Q$ be the pressure space $=L_0^2(\Omega)$. Let $X^h$ be the finite element velocity space of continuous piecewise polynomials based on conforming partition of $\Omega$ into elements, denoted $\Dt$, $X^h\subset X$. Assume $X^h$ satisfies the approximation properties:
\begin{equation}\label{approx-prop}
\begin{aligned}
    \inf_{v\in X^h}\|u-v\|&\leq Ch^{m+1}|u|_{m+1},\quad u\in H^{m+1}(\Omega)^d,\\
    \inf_{v\in X^h}\|\nabla(u-v)\|&\leq Ch^{m}|u|_{m+1},\quad u\in H^{m+1}(\Omega)^d.
    %\inf_{q\in Q^h}\|p-q\|&\leq Ch^m|p|_m,\quad p\in H^m(\Omega).
\end{aligned}
\end{equation}
The space $H^{-1}(\Omega)$ denotes the dual space of bounded linear functionals defined on $H_0^1(\Omega)%=\{v\in H^1(\Omega): v=0\ \text{on}\  \partial\Omega\}
$. This space is equipped with the norm:
\begin{equation*}
    \|f\|_{-1}=\sup_{0\neq v\in X}\frac{(f,v)}{\|\nabla v\|}.%\quad \forall f\in H^{-1}(\Omega).
\end{equation*}
Denote $a(u,v,w)=\sum_\Dt LocTol_\Dt^{-1}\int_\Dt|\nabla\cdot u|^2\nabla\cdot v\nabla\cdot w\ dx$ for any $u,v,w\in X$.

\begin{lemma}(Useful inequalities see p.7 \cite{cfdbook})
\iffalse
The $L^2$ inner product satisfies the Cauchy-Schwarz, Young inequalities and polarization relation: for any $u,v \in L^2(\Omega)$
\begin{align}
    (u,v) &\leq \|u\|\|v\|,\ \text{and}\ 
    (u,v) \leq\frac{\delta}{2}\|u\|^2+\frac{1}{2\delta}\|v\|^2,\quad\text{for any}\ \delta, 0<\delta<\infty. \label{young}\\
     (u,v)&=\frac{1}{2}\|u\|^2+\frac{1}{2}\|v\|^2-\frac{1}{2}\|u-v\|^2,\quad\forall\ u,v\in L^2(\Omega)\label{polar}
\end{align}
\fi
(H\"{o}lder's and Young's inequalities)
For any $\delta$, $0 < \delta < \infty$ and $\frac{1}{p}+\frac{1}{q}=1, 1\leq p,q \leq \infty$,
\begin{equation}\label{holder+young}
\begin{aligned}
    (u,v)\leq\|u\|_{L^p}\|v\|_{L^q},\ \text{and}\ 
    (u,v)\leq \frac{\delta}{p}\|u\|_{L^p}^p+\frac{\delta^{-q/p}}{q}\|v\|_{L^q}^q.
\end{aligned}
\end{equation}
\end{lemma}

On each mesh element $\Dt$ denote $(\phi,\psi)_\Dt=\int_\Dt \phi\cdot\psi\ dx$.
The nonlinear term satisfies the following, often called Strong Monotonicity, and Local Lipschitz continuity. 
\begin{lemma}(Strong Monotonicity and Local Lipschitz continuity) Let $u,v,w\in X$, on each mesh element $\Dt$, then there exist constants $C_1, C_2$ such that the following inequalities hold:
\begin{align}\label{monotone}
    (|\nabla\cdot u|^2\nabla\cdot u-|\nabla\cdot w|^2\nabla\cdot w,\nabla\cdot(u-w))_\Dt\geq C_1\|\nabla\cdot(u-w)\|_{L^4(\Dt)}^4, \\
    (|\nabla\cdot u|^2\nabla\cdot u-|\nabla\cdot w|^2\nabla\cdot w,\nabla\cdot v)_\Dt\leq C_2r^2\|\nabla\cdot(u-w)\|_{L^4(\Dt)}\|\nabla\cdot v\|_{L^4(\Dt)}, \label{llc}\\
    \text{where}\ r=\max\{\|\nabla\cdot u\|_{L^4(\Dt)},\|\nabla\cdot w\|_{L^4(\Dt)}\}. \nonumber
\end{align}
\end{lemma}
\begin{proof}
(of Local Lipschitz continuity)
\begin{align*}
    &(|\nabla\cdot u|^2\nabla\cdot u-|\nabla\cdot w|^2\nabla\cdot w,\nabla\cdot v)_\Dt\\
    &=(|\nabla\cdot u|^2\nabla\cdot u-|\nabla\cdot u|^2\nabla\cdot w,\nabla\cdot v)_\Dt
    +(|\nabla\cdot u|^2\nabla\cdot w-|\nabla\cdot w|^2\nabla\cdot w,\nabla\cdot v)_\Dt \\
    &=\int_\Dt |\nabla\cdot u|^2\nabla\cdot(u-w)\nabla\cdot v\ dx+\int_\Dt\nabla\cdot w(\nabla\cdot u+\nabla\cdot w)(\nabla\cdot u-\nabla\cdot w)\nabla\cdot v\ dx \\
    &=\int_\Dt \nabla\cdot(u-w)\nabla\cdot v(|\nabla\cdot u|^2+\nabla\cdot u\nabla\cdot w+|\nabla\cdot w|^2)\ dx \\
    &\leq \int_\Dt |\nabla\cdot(u-w)||\nabla\cdot v|(|\nabla\cdot u|+|\nabla\cdot w|)^2\ dx \\
    &\leq \|\nabla\cdot(u-w)\|_{L^4(\Dt)}\|\nabla\cdot v\|_{L^4(\Dt)}\left(\int_\Dt (|\nabla\cdot u|+|\nabla\cdot w|)^4\ dx\right)^{1/2},
\end{align*}
Denote $r=\max(\|\nabla\cdot u\|_{L^4(\Dt)},\|\nabla\cdot w\|_{L^4(\Dt)})$, then we have
\begin{equation*}
    (|\nabla\cdot u|^2\nabla\cdot u-|\nabla\cdot w|^2\nabla\cdot w,\nabla\cdot v)_\Dt\leq C_2r^2\|\nabla\cdot(u-w)\|_{L^4(\Dt)}\|\nabla\cdot v\|_{L^4(\Dt)}.
\end{equation*}
The proof of Strong Monotonicity follows similarly to the p-Laplacian in Barrett and Liu \cite{barrett1993finite}, Glowinski and Marroco \cite{p-laplacian}, omit the part here.
\end{proof}

\section{Analysis} In this section, we derived stability bounds for both new penalty methods (PP \eqref{vform} and EP \eqref{vform2})  and error estimates for the pointwise penalty (PP) method \eqref{vform}.
%Stability
\subsection{Stability} First, we consider the elementwise penalty (EP) method \eqref{vform2}. Recall that $LocTol_\Dt=\frac{1}{2}TOL^2\frac{|\Dt|}{|\Omega|}$.
\begin{theorem}
Suppose $\mathcal{T}^h$ be a mesh of $\Omega$ and $\Dt$ denote a mesh element in $\mathcal{T}^h$, the solution to \eqref{vform2} is stable, and the following stability bound holds
\begin{equation*}
    \frac{\nu}{2}\|\nabla u^h_\epsilon\|^2+\sum_\Dt\frac{1}{LocTol_\Dt}\|\nabla\cdot u^h_\epsilon\|_\Dt^4\leq\frac{1}{2\nu}\|f\|_{-1}^2.
\end{equation*}
\end{theorem}
\begin{proof}
Take $v^h=u_\epsilon^h$ in \eqref{vform2}:
\begin{align*}
    \nu\|\nabla u_\epsilon^h\|^2+\sum_\Dt\frac{1}{LocTol_\Dt}\|\nabla\cdot u_\epsilon^h\|_\Dt^2\|\nabla\cdot u_\epsilon^h\|_\Dt^2=\int_\Omega f\cdot u_\epsilon^h\ dx,
\end{align*}
As $(f,u_\epsilon^h)\leq \|f\|_{-1}\|\nabla u^h_\epsilon\|$ and apply H\"{o}lder's and Young's inequalities \eqref{holder+young}:
\begin{equation*}
    \nu\|\nabla u^h_\epsilon\|^2+\sum_\Dt\frac{1}{LocTol_\Dt}\|\nabla\cdot u^h_\epsilon\|_\Dt^4\leq\frac{1}{2\nu}\|f\|_{-1}^2+\frac{\nu}{2}\|\nabla u^h_\epsilon\|^2.
\end{equation*}
Combine similar terms and the claimed stability bound then follows.
\end{proof}\\
From Theorem 3.1, we have the following proposition.
\begin{proposition}
Let $N$ denote the number of elements $\Dt$ in mesh $\mathcal{T}^h$ and TOL denote the global tolerance, then the solution $u_\epsilon^h$ to \eqref{vform2} satisfy
\begin{equation*}
    \|\nabla\cdot u^h_\epsilon\|^4\leq \left(\frac{N\cdot\max|\Dt|}{4\nu|\Omega|}\right)TOL^2\|f\|_{-1}^2.
\end{equation*}
\end{proposition}
\begin{proof}
From Theorem 3.1, we have
\begin{equation*}
\sum_\Dt\frac{1}{LocTol_\Dt}\|\nabla\cdot u^h_\epsilon\|_\Dt^4\leq\frac{1}{2\nu}\|f\|_{-1}^2.
\end{equation*}
Recall $LocTol_\Dt=\frac{1}{2}TOL^2\frac{|\Dt|}{|\Omega|}$,
\begin{align*}
    \sum_\Dt\frac{2}{TOL^2}\frac{|\Omega|}{|\Dt|}\|\nabla\cdot u^h_\epsilon\|_\Dt^4\leq\frac{1}{2\nu}\|f\|_{-1}^2,\\
    \sum_\Dt\frac{1}{|\Dt|}\|\nabla\cdot u^h_\epsilon\|_\Dt^4\leq\frac{TOL^2}{4\nu|\Omega|}\|f\|_{-1}^2,\\
    \frac{1}{\max|\Dt|}\sum_\Dt\left(\int_\Dt|\nabla\cdot u_\epsilon^h|^2\ dx\right)^2\leq\frac{TOL^2}{4\nu|\Omega|}\|f\|_{-1}^2,\\
    \sum_\Dt\left(\int_\Dt|\nabla\cdot u_\epsilon^h|^2\ dx\right)^2\leq\frac{\max|\Dt|}{|\Omega|}\frac{TOL^2}{4\nu}\|f\|_{-1}^2.
\end{align*}
Using the Cauchy Schwartz inequality:
\begin{align*}
    \frac{1}{N}\left(\sum_\Dt\int_\Dt|\nabla\cdot u^h_\epsilon|^2\ dx\right)^2\leq\sum_\Dt\left(\int_\Dt|\nabla\cdot u^h_\epsilon|^2\ dx\right)^2\leq\frac{\max|\Dt|}{|\Omega|}\frac{TOL^2}{4\nu}\|f\|_{-1}^2.
\end{align*}
Then the result follows.
\end{proof}

Next, we consider the pointwise penalty (PP) method \eqref{vform}. Recall that $LocTol_\Dt=\frac{1}{2}TOL^2\frac{|\Dt|}{|\Omega|}$.
\begin{theorem}
Suppose $\mathcal{T}^h$ be a mesh of $\Omega$ and $\Dt$ denote the mesh element in $\mathcal{T}^h$, the solution to \eqref{vform} is stable, and the following stability bound holds
\begin{equation*}
    \frac{\nu}{2}\|\nabla u_\epsilon^h\|^2+\sum_\Dt\frac{1}{LocTol_\Dt}\|\nabla\cdot u_\epsilon^h\|_{L^4(\Dt)}^4\leq\frac{1}{2\nu}\|f\|_{-1}^2.
\end{equation*}
\end{theorem}
\begin{proof}
Take $v^h=u_\epsilon^h$ in \eqref{vform}:
\begin{align*}
    \nu\|\nabla u_\epsilon^h\|^2+\sum_\Dt\frac{1}{LocTol_\Dt}\int_\Dt |\nabla\cdot u_\epsilon^h|^4\ dx=\int_\Omega f\cdot u_\epsilon^h\ dx,
\end{align*}
As $(f,u^h_\epsilon)\leq\|f\|_{-1}\|\nabla u^h_\epsilon\|$ and apply H\"{o}lder's and Young's inequalities \eqref{holder+young}:
\begin{align*}
    \nu\|\nabla u_\epsilon^h\|^2+\sum_\Dt\frac{1}{LocTol_\Dt}\|\nabla\cdot u_\epsilon^h\|_{L^4(\Dt)}^4\leq \frac{1}{2\nu}\|f\|_{-1}^2+\frac{\nu}{2}\|\nabla u_\epsilon^h\|^2.
\end{align*}
Combine similar terms and the claimed stability bound then follows.
\end{proof}\\
Directly from the result of Theorem 3.3, we have the following proposition.
\begin{proposition}
Let TOL denote the global tolerance, then the solution $u^h_\epsilon$ to \eqref{vform} satisfy
\begin{equation*}
    \|\nabla\cdot u^h_\epsilon\|^4_{L^4}\leq\left(\frac{\max|\Dt|}{4\nu|\Omega|}\right)TOL^2\|f\|_{-1}^2.
\end{equation*}
\end{proposition}
\begin{proof}
From Theorem 3.3, there holds
\begin{equation*}
    \sum_\Dt\frac{1}{LocTol_\Dt}\|\nabla\cdot u^h_\epsilon\|^4_{L^4(\Dt)}\leq \frac{1}{2\nu}\|f\|^2_{-1}.
\end{equation*}
Recall $LocTol_\Dt=\frac{1}{2}TOL^2\frac{|\Dt|}{|\Omega|}$,
\begin{align*}
    \sum_\Dt \frac{2}{TOL^2}\frac{|\Omega|}{|\Dt|}\|\nabla\cdot u^h_\epsilon\|^4_{L^4(\Dt)}\leq\frac{1}{2\nu}\|f\|_{-1}^2,\\
    \sum_\Dt\frac{1}{|\Dt|}\|\nabla\cdot u^h_\epsilon\|^4_{L^4(\Dt)}\leq\frac{TOL^2}{4\nu|\Omega|}\|f\|_{-1}^2,\\
    \frac{1}{\max|\Dt|}\sum_\Dt\left(\int_\Dt|\nabla\cdot u^h_\epsilon|^4\ dx\right)^1\leq\frac{TOL^2}{4\nu|\Omega|}\|f\|_{-1}^2.
\end{align*}
Then the result follows.
\end{proof}
%++++++++++++++old stability+++++++++++++++++
%error analysis version 2
\subsection{Error analysis} 
% version 4 error analysis
%{\cc newest error between continuous Stokes and discretized penalized Stokes}\\
We consider the error between continuous Stokes problem \eqref{stokes} and discretized pointwise penalized (PP) Stokes problem \eqref{vform}. Recall $Q:=\{q\in L^2(\Omega):\int_\Omega q\ dx=0\}$. The variational form of the Stokes problem \eqref{stokes} is: 

Find $(u,p)\in (X,Q)$ such that
\begin{equation}\label{v-stokes}
\begin{aligned}
    \int_\Omega\nu\nabla u:\nabla v\ dx-\int_\Omega p(\nabla\cdot v)\ dx&=\int_\Omega f\cdot v\ dx\quad\text{for all}\ v\in X, \\
    \text{and}\quad\int_\Omega (\nabla\cdot u) q\ dx&=0\quad\text{for all}\ q\in Q.
\end{aligned}
\end{equation}
\begin{theorem}
Let $(u,p)$ be a solution to the Stokes problem \eqref{v-stokes} and $u_\epsilon^h$ be the solution of the penalty approximation \eqref{vform}. Let d denote the dimension of $\Omega$ and $C_1, C_2$ be two constants defined as in \eqref{monotone} and \eqref{llc}. TOL denote the global tolerance and $LocTol_\Dt$ be the local tolerance for each element $\Dt$ in mesh $\mathcal{T}^h$. Then it follows that
\begin{align*}
   \nu\|\nabla (u-u^h_\epsilon)\|^2+C_1\sum_\Dt LocTol_\Dt^{-1}\|\nabla\cdot (u-u^h_\epsilon)\|_{L^4(\Dt)}^4\leq\inf_{v^h\in X^h}C(C_1,C_2)\sum_\Dt LocTol_\Dt^{-1}
    \|\nabla\cdot(u-v^h)\|^4_{L^4(\Dt)}\\ +C(\nu)h^{2m-2}\|u\|_{H^{m+1}(\Omega)}^2
    +h^2\|p\|^2+C\nu^{-1/4}\|f\|^{1/2}_{-1}TOL^{1/2}(\max |\Dt|)^{1/4}\|p\|^2.
\end{align*}
\end{theorem}
\begin{remark}
 If $X^h$ has a divergence free subspace with good approximation properties, the first term of the RHS of the estimate in Theorem 3.5 vanishes.
\end{remark}
\begin{proof}
As $a(u,u,v^h)=\sum_\Dt LocTol_\Dt^{-1}\int_\Dt|\nabla\cdot u|^2\nabla\cdot u\nabla\cdot v^h\ dx$ and $\nabla\cdot u=0$, so $a(u,u,v^h)=0$. From \eqref{v-stokes}, adding $a(u,u,v)$ to the left-hand-side :
\begin{equation*}
    \nu(\nabla u,\nabla v)-(p,\nabla\cdot v)+a(u,u,v)=(f,v),\quad\forall\ v\in X.
\end{equation*}
Subtract \eqref{vform} and let $v=v^h$:
\begin{equation*}
    \nu(\nabla(u-u_\epsilon^h),\nabla v^h)+a(u,u,v^h)-a(u_\epsilon^h,u_\epsilon^h,v^h)=(p,\nabla\cdot v^h).
\end{equation*}
Denote $e=u-u^h_\epsilon$, let $\forall\ \Tilde{u}\in X^h, \eta=u-\Tilde{u}$ and $\phi^h=u^h_\epsilon-\Tilde{u}$, then $e=\eta-\phi^h$, the error equation becomes:
\begin{align*}
    \nu(\nabla\eta,\nabla v^h)+a(u,u,v^h)-a(\Tilde{u},\Tilde{u},v^h)\\
    =\nu(\nabla\phi^h,\nabla v^h)+a(u_\epsilon^h,u_\epsilon^h,v^h)-a(\Tilde{u},\Tilde{u},v^h)+(p,\nabla\cdot v^h),
\end{align*}
Letting $v^h=\phi^h$, the error equation becomes:
\begin{equation*}
    \nu(\nabla\phi^h,\nabla \phi^h)+a(u^h_\epsilon,u^h_\epsilon,\phi^h)-a(\Tilde{u},\Tilde{u},\phi^h)=\nu(\nabla\eta,\nabla\phi^h)+a(u,u,\phi^h)-a(\Tilde{u},\Tilde{u},\phi^h)-(p,\nabla\cdot\phi^h).
\end{equation*}
Apply Strong Monotonicity \eqref{monotone} to $a(u^h_\epsilon,u^h_\epsilon,\phi^h)-a(\Tilde{u},\Tilde{u},\phi^h)$:
\begin{align*}
    &a(u^h_\epsilon,u^h_\epsilon,\phi^h)-a(\Tilde{u},\Tilde{u},\phi^h)\\
    &=\sum_\Dt\frac{1}{LocTol_\Dt}\int_\Dt(|\nabla\cdot u^h_\epsilon|^2\nabla\cdot u^h_\epsilon-|\nabla\cdot\Tilde{u}|^2\nabla\cdot\Tilde{u})\nabla\cdot(u^h_\epsilon-\Tilde{u})\ dx\\
    &\geq \sum_\Dt\frac{1}{LocTol_\Dt}C_1\int_\Dt|\nabla\cdot(u^h_\epsilon-\Tilde{u})|^4\ dx.
\end{align*}
Apply Local Lipschitz continuity \eqref{llc} to $a(u,u,\phi^h)-a(\Tilde{u},\Tilde{u},\phi^h)$:
\begin{align*}
   &a(u,u,\phi^h)-a(\Tilde{u},\Tilde{u},\phi^h)\\
   &=\sum_\Dt\frac{1}{LocTol_\Dt}\int_\Dt (|\nabla\cdot u|^2\nabla\cdot u-|\nabla\cdot\Tilde{u}|^2\nabla\cdot\Tilde{u})\nabla\cdot\phi^h\ dx\\
   &\leq\sum_\Dt\frac{1}{LocTol_\Dt} C_2 r_\Dt^2\left(\int_\Dt |\nabla\cdot(u-\Tilde{u})|^4\ dx\right)^{1/4}\left(\int_\Dt |\nabla\cdot\phi^h|^4\right)^{1/4}\\
   &\text{where}\ r_\Dt=\max\{\|\nabla\cdot u\|_{L^4(\Dt)},\|\nabla\cdot\Tilde{u}\|_{L^4(\Dt)}\}=\|\nabla\cdot\Tilde{u}\|_{L^4(\Dt)}.
\end{align*}
Then the error equation becomes
\begin{align*}
    \nu\|\nabla\phi^h\|^2+\sum_\Dt\frac{C_1}{LocTol_\Dt}\|\nabla\cdot\phi^h\|_{L^4(\Dt)}^4\leq\nu(\nabla\eta,\nabla\phi^h)\\
    +\sum_\Dt\frac{C_2r_\Dt^2}{LocTol_\Dt}\|\nabla\cdot\eta\|_{L^4(\Dt)}\|\nabla\cdot\phi^h\|_{L^4(\Dt)}-(p,\nabla\cdot\phi^h).
\end{align*}
Apply H\"{o}lder's and Young's inequality \eqref{holder+young} with $p=4, q=4/3$:
\begin{align*}
    &\nu\|\nabla\phi^h\|^2+\sum_\Dt\frac{C_1}{LocTol_\Dt}\|\nabla\cdot\phi^h\|_{L^4(\Dt)}^4\\
    &\leq\frac{\nu}{2}\|\nabla\eta\|^2+\frac{\nu}{2}\|\nabla\phi^h\|^2+\sum_\Dt\left(\frac{C_1^{1/4}}{LocTol_\Dt^{1/4}}\|\nabla\cdot\phi^h\|_{L^4(\Dt)}\right)\left(\frac{C_2r_\Dt^2}{C_1^{1/4}LocTol_\Dt^{3/4}}\|\nabla\cdot\eta\|_{L^4(\Dt)}\right)-(p,\nabla\cdot\phi^h)\\
    &\leq\frac{\nu}{2}\|\nabla\eta\|^2+\frac{\nu}{2}\|\nabla\phi^h\|^2-(p,\nabla\cdot\phi^h)\\
    &+\left(\sum_\Dt\frac{C_1}{LocTol_\Dt}\|\nabla\cdot\phi^h\|_{L^4(\Dt)}^4\right)^{1/4}\left(\sum_\Dt\frac{C_2^{4/3}r_\Dt^{8/3}}{C_1^{1/3}LocTol_\Dt}\|\nabla\cdot\eta\|_{L^4(\Dt)}^{4/3}\right)^{3/4}\\
    &\leq \frac{\nu}{2}\|\nabla\eta\|^2+\frac{\nu}{2}\|\nabla\phi^h\|^2-(p,\nabla\cdot\phi^h)\\
    &+\frac{\delta}{4}\left(\sum_\Dt\frac{C_1}{LocTol_\Dt}\|\nabla\cdot\phi^h\|_{L^4(\Dt)}^4\right)^1+\frac{\delta^{-1/3}}{4/3}\left(\sum_\Dt\frac{C_2^{4/3}r_\Dt^{8/3}}{C_1^{1/3}LocTol_\Dt}\|\nabla\cdot\eta\|_{L^4(\Dt)}^{4/3}\right)^1.
\end{align*}
Letting $\delta=2$ and combining similar terms gives
\begin{equation*}
    \frac{\nu}{2}\|\nabla\phi^h\|^2+\frac{1}{2}\sum_\Dt\frac{C_1}{LocTol_\Dt}\|\nabla\cdot\phi^h\|_{L^4(\Dt)}^4\leq\frac{\nu}{2}\|\nabla\eta\|^2+\frac{3}{4\sqrt[3]{2}}\sum_\Dt\frac{C_2^{4/3}r_\Dt^{8/3}}{C_1^{1/3}LocTol_\Dt}\|\nabla\cdot\eta\|_{L^4(\Dt)}^{4/3}   -(p,\nabla\cdot\phi^h).
\end{equation*}
Consider the last term of the error equation inspired by the proof of Falk \cite{falk1976finite}:
\begin{align*}
    (p,\nabla\cdot\phi^h)&=(p,\nabla\cdot(u^h_\epsilon-\Tilde{u}))\\
    &=(p,\nabla\cdot u^h_\epsilon)+(p,\nabla\cdot (u-\Tilde{u}))\\
    &\leq\sum_\Dt\int_\Dt p\nabla\cdot u^h_\epsilon\ dx+\frac{h^2}{2}\|p\|^2+\frac{1}{2h^2}\|\nabla\cdot\eta\|^2\\
    &\leq \sum_\Dt\int_\Dt \frac{1}{LocTol_\Dt^{1/4}}|\nabla\cdot u^h_\epsilon| LocTol_\Dt^{1/4}|p|\ dx+\frac{h^2}{2}\|p\|^2+\frac{1}{2h^2}\|\nabla\cdot\eta\|^2\\
    &\leq\sum_\Dt\left(\int_\Dt\frac{1}{LocTol_\Dt}|\nabla\cdot u^h_\epsilon|^4\ dx\right)^{1/4}\left(\int_\Dt LocTol_\Dt^{1/3}|p|^{4/3}\ dx\right)^{3/4}+\frac{h^2}{2}\|p\|^2+\frac{1}{2h^2}\|\nabla\cdot\eta\|^2\\
    &\leq\left(\sum_\Dt\int_\Dt\frac{1}{LocTol_\Dt}|\nabla\cdot u^h_\epsilon|^4\ dx\right)^{1/4}\left(\sum_\Dt\int_\Dt LocTol_\Dt^{1/3}|p|^{4/3}\ dx\right)^{3/4}+\frac{h^2}{2}\|p\|^2+\frac{1}{2h^2}\|\nabla\cdot\eta\|^2\\
    &\leq\left(\sum_\Dt\frac{1}{LocTol_\Dt}\|\nabla\cdot u^h_\epsilon\|^4_{L^4(\Dt)}\right)^{1/4}\left(\sum_\Dt\left(\int_\Dt LocTol_\Dt^{1}\right)^{1/3}\left(\int_\Dt |p|^2\right)^{2/3}\right)^{3/4}\\
    &+\frac{h^2}{2}\|p\|^2+\frac{1}{2h^2}\|\nabla\cdot\eta\|^2\\
    &\leq\left(\sum_\Dt\frac{1}{LocTol_\Dt}\|\nabla\cdot u^h_\epsilon\|^4_{L^4(\Dt)}\right)^{1/4}\left(\left(\sum_\Dt\int_\Dt LocTol_\Dt\right)^{1/3}\left(\sum_\Dt\int_\Dt|p|^2\right)^{2/3}\right)^{3/4}\\
    &+\frac{h^2}{2}\|p\|^2+\frac{1}{2h^2}\|\nabla\cdot\eta\|^2\\
    &=\left(\sum_\Dt\frac{1}{LocTol_\Dt}\|\nabla\cdot u^h_\epsilon\|_{L^4(\Dt)}^4\right)^{1/4}\left(\sum_\Dt|\Dt|LocTol_\Dt\right)^{1/4}\|p\|+\frac{h^2}{2}\|p\|^2+\frac{1}{2h^2}\|\nabla\cdot\eta\|^2.
\end{align*}
By the stability bound,
\begin{equation*}
    \frac{\nu}{2}\|\nabla u_\epsilon^h\|^2+\sum_\Dt\frac{1}{LocTol_\Dt}\|\nabla\cdot u_\epsilon^h\|_{L^4(\Dt)}^4\leq\frac{1}{2\nu}\|f\|_{-1}^2.
\end{equation*}
Thus,
\begin{equation*}
    (p,\nabla\cdot\phi^h)\leq C\nu^{-1/4}\|f\|_{-1}^{1/2}\left(\sum_\Dt|\Dt|LocTol_\Dt\right)^{1/4}\|p\|+\frac{h^2}{2}\|p\|^2+\frac{1}{2h^2}\|\nabla\cdot\eta\|^2.
\end{equation*}
Plug back to the error equation:
\begin{align*}
    \frac{\nu}{2}\|\nabla\phi^h\|^2+\sum_\Dt\frac{C_1}{2LocTol_\Dt}\|\nabla\cdot\phi^h\|_{L^4(\Dt)}^4\leq\frac{\nu}{2}\|\nabla\eta\|^2+\sum_\Dt\frac{3C_2^{4/3}r_\Dt^{8/3}}{4\sqrt[3]{2}C_1^{1/3}LocTol_\Dt}\|\nabla\cdot\eta\|_{L^4(\Dt)}^{4/3}\\
    +C\nu^{-1/4}\|f\|_{-1}^{1/2}\left(\sum_\Dt|\Dt|LocTol_\Dt\right)^{1/4}\|p\|+\frac{h^2}{2}\|p\|^2+\frac{1}{2h^2}\|\nabla\cdot\eta\|^2,
\end{align*}
where
\begin{align*}
    \left(\sum_\Dt|\Dt|LocTol_\Dt\right)^{1/4}:=\left(\sum_\Dt|\Dt|\frac{1}{2}\frac{TOL^2}{|\Omega|}|\Dt|\right)^{1/4}=\left(\frac{TOL^2}{2|\Omega|}\sum_\Dt|\Dt|^2\right)^{1/4}\\
    \leq \frac{TOL^{1/2}}{2^{1/4}}\left(\frac{1}{|\Omega|}\max|\Dt|\sum_\Dt|\Dt|\right)^{1/4}=TOL^{1/2}\left(\frac{\max |\Dt|}{2}\right)^{1/4}.
\end{align*}
Apply triangle inequality: $\|e\|\leq\|\eta\|+\|\phi^h\|$
\begin{align*}
    \nu\|\nabla e\|^2+\sum_\Dt\frac{C_1}{LocTol_\Dt}\|\nabla\cdot e\|_{L^4(\Dt)}^4\leq \inf_{v^h\in X^h}\Big\{\nu\|\nabla(u-v^h)\|^2+h^{-2}\|\nabla\cdot(u-v^h)\|^2\\
    +C(C_1,C_2)\sum_\Dt LocTol_\Dt^{-1}\left( r_\Dt^{8/3}\|\nabla\cdot(u-v^h)\|_{L^4(\Dt)}^{4/3}+\|\nabla\cdot(u-v^h)\|_{L^4(\Dt)}^4\right)
    \Big\}\\
    +h^2\|p\|^2+C\nu^{-1/4}\|f\|^{1/2}_{-1}TOL^{1/2}(\max |\Dt|)^{1/4}\|p\|^2,
\end{align*}
where
\begin{align*}
    \sum_\Dt LocTol_\Dt^{-1}r_\Dt^{8/3}\|\nabla\cdot(u-v^h)\|_{L^4(\Dt)}^{4/3}&=\sum_\Dt LocTol_\Dt^{-1}\|\nabla\cdot v^h\|_{L^4(\Dt)}^{8/3}\|\nabla\cdot(u-v^h)\|_{L^4(\Dt)}^{4/3}\\
    &=\sum_\Dt LocTol_\Dt^{-1}\|\nabla\cdot (u-v^h)\|_{L^4(\Dt)}^4.
   %&\leq (\min LocTol_\Dt)^{-1}Ch^{-d}\left[\sum_\Dt\|\nabla\cdot(u-v^h)\|_{L^2(\Dt)}^2\right]^{2}\\
    %&=(\min LocTol_\Dt)^{-1}Ch^{-d}\left[\sum_\Dt\int_\Dt|\nabla\cdot(u-v^h)|^2\ dx\right]^2\\
    %&\leq(\min LocTol_\Dt)^{-1}Ch^{-d}\|\nabla\cdot(u-v^h)\|^4,
\end{align*}
The error satisfies
\begin{align*}
    \nu\|\nabla e\|^2+\sum_\Dt\frac{C_1}{LocTol_\Dt}\|\nabla\cdot e\|_{L^4(\Dt)}^4\leq \inf_{v^h\in X^h}\Big\{\nu\|\nabla(u-v^h)\|^2+h^{-2}\|\nabla\cdot(u-v^h)\|^2\\
    +C(C_1,C_2)%\left(h^{-d/3}\|\nabla\cdot v^h\|_{L^4}^{8/3}\|\nabla\cdot(u-v^h)\|^{4/3}+
    \sum_\Dt LocTol_\Dt^{-1}\|\nabla\cdot(u-v^h)\|_{L^4(\Dt)}^4
    \Big\}
    +h^2\|p\|^2+C\nu^{-1/4}\|f\|^{1/2}_{-1}TOL^{1/2}(\max |\Dt|)^{1/4}\|p\|^2.
\end{align*}
As $\|\nabla\cdot(u-v^h)\|\leq\|\nabla(u-v^h)\|$
\begin{align*}
    \nu\|\nabla e\|^2+\sum_\Dt\frac{C_1}{LocTol_\Dt}\|\nabla\cdot e\|_{L^4(\Dt)}^4\leq \inf_{v^h\in X^h}\Big\{C(\nu)(1+h^{-2})\|\nabla(u-v^h)\|^2\\
    +C(C_1,C_2)\sum_\Dt LocTol_\Dt^{-1}\|\nabla\cdot(u-v^h)\|_{L^4(\Dt)}^4
    \Big\}
    +h^2\|p\|^2+C\nu^{-1/4}\|f\|^{1/2}_{-1}TOL^{1/2}(\max |\Dt|)^{1/4}\|p\|^2.
\end{align*}
Using the approximation properties \eqref{approx-prop} of the spaces $X^h$
\begin{align*}
    \nu\|\nabla e\|^2+\sum_\Dt\frac{C_1}{LocTol_\Dt}\|\nabla\cdot e\|_{L^4(\Dt)}^4\leq \inf_{v^h\in X^h}C(C_1,C_2)\sum_\Dt LocTol_\Dt^{-1}\|\nabla\cdot(u-v^h)\|^4_{L^4(\Dt)}\\+C(\nu)h^{2m-2}\|u\|_{H^{m+1}(\Omega)}^2
    +h^2\|p\|^2++C\nu^{-1/4}\|f\|^{1/2}_{-1}TOL^{1/2}(\max |\Dt|)^{1/4}\|p\|^2.
\end{align*}
\end{proof}

%=================================================

%====================================================
\section{Algorithm}
This section presents the algorithms to implement the elementwise variable $\epsilon$ elementwise penalty (EP) method \eqref{vform2} introduced in Section 1. The following Algorithm 1 is for Stokes problem.
\begin{algorithm}[H]
\SetAlgoLined
Given tolerance TOL, epsilon lower bound LowerEps and mesh $\mathcal{T}$, MaxIter=10\\
\textit{Compute} on each element triangle LocTol$_\Dt =\frac{1}{2}\frac{TOL^2}{|\Omega|}|\Dt|$ \\
\textit{Set} $\epsilon_\Dt=1$ \\
Solve for $u^h$ using penalty method: find $u^h\in X^h$ such that
\begin{align*}
    \nu(\nabla u_\epsilon^h,\nabla v^h)+\sum_\Dt\int_\Dt \epsilon_\Dt^{-1}\nabla\cdot u_\epsilon^h\nabla\cdot v^h\ dx=(f,v^h)\quad \forall v^h\in X^h
\end{align*}
\While{iteration $\leq$ MaxIter and retry=true}{
Loop over all triangle elements $\Dt\in \mathcal{T}$ \\
\quad\quad\textit{Compute} estimator for each triangle
    \begin{equation*}
        est_\Dt=\int_\Dt |\nabla\cdot u_\epsilon^h|^2\ dx
    \end{equation*}
    \If{$est_\Dt>LocTol_\Dt$}{
        $r=\frac{LocTol_\Dt}{est_\Dt}$ \;
        $\epsilon_\Dt\gets \max(LowerEps, r\times\epsilon_\Dt)$\;
        retry=true;\\}
 REPEAT step;\\
}
Recover pressure $p$ if needed
\begin{align*}
     p_\Dt=-\frac{1}{\epsilon_\Dt}\nabla\cdot u^h_\epsilon
\end{align*}
\caption{Elementwise variable $\epsilon$ penalty (EP) method for Stokes}
\end{algorithm}
\begin{remark}
 We need to set a maximum number of iteration $MaxIter$ in the loop to avoid the program run infinitely. But this may lead to the situation that $est_\Dt\geq LocTol_\Dt$ local tolerance is not satisfied. However, our ultimate goal is $\|\nabla\cdot u^h\|<TOL$ no matter local tolerance is satisfied or not. 
\end{remark}

We also want to test the elementwise variable $\epsilon$ penalty method on the unsteady Navier-Stokes equation.
For time-dependent problem \eqref{nse}, there are two options:
\begin{enumerate}
    \item use $\|\nabla\cdot u_\epsilon^h\|_\Dt$ from the previous time step, adjust $\epsilon$ and do not repeat the current time-step,
    \item for each time-step, repeat using $\epsilon_{new}$ and loop until tolerance or maximum iteration is reached.
\end{enumerate}
Since this is a new algorithm, we do not know which is better. We still need to do further research, and Algorithm 2 follows.
\begin{algorithm}[h!]
\SetAlgoLined
Given tolerance TOL, epsilon lower bound LowerEps and mesh $\mathcal{T}^h$, final time $T_{final}$, time-step $\Dt t$, initial condition $u_0(x)$\\
\textit{Compute} on each element triangle LocTol$_\Dt =\frac{1}{2}\frac{TOL^2}{|\Omega|}|\Dt|$ \\
\textit{Set} $\epsilon_{\Dt,1}=1,t_0=0$; \\
\While{$t<T_{final}$}{
\textit{Update} $t_{n+1}=t_n+\Dt t$;\\
Given $u_{\epsilon,n}^h$, solve for $u_{\epsilon,n+1}^h$ using penalty method: find $u_{\epsilon,n+1}^h\in X^h$ such that
\begin{gather*}
    (\frac{u^h_{\epsilon,n+1}-u^h_{\epsilon,n}}{\Dt t},v^h)+(u_{\epsilon,n}^h\cdot\nabla u^h_{\epsilon,n+1},v^h)+\frac{1}{2}\left((\nabla\cdot u_{\epsilon,n}^h)u^h_{\epsilon,n+1},v^h\right)+\nu(\nabla u_{\epsilon,n+1}^h,\nabla v^h) \\
    +\sum_\Dt\int_\Dt \epsilon_{\Dt,n+1}^{-1}\nabla\cdot u_{\epsilon,n+1}^h\nabla\cdot v^h\ dx=(f^{n+1},v^h)\quad \forall v^h\in X^h
\end{gather*}
Loop over all triangle elements $\Dt\in \mathcal{T}^h$ \\
\quad\quad\textit{Compute} estimator for each triangle
    \begin{equation*}
        est_\Dt=\int_\Dt |\nabla\cdot u_{\epsilon,n+1}^h|^2\ dx
    \end{equation*}
\quad\quad\textit{Update} $\epsilon_\Dt$:
    \begin{align*}
        r&=\frac{LocTol_\Dt}{est_\Dt}, \\
        \epsilon_{\Dt,n+2}&\gets \max(LowerEps, r\times\epsilon_{\Dt,n+1}), \\
        \text{retry}&=\text{false};
    \end{align*}
\quad\quad\textit{Recover} pressure $p$ if needed
\begin{align*}
     p_{\Dt,n+1}=-\frac{1}{\epsilon_{\Dt,n+1}}\nabla\cdot u_{\epsilon,n+1}^h.
\end{align*}       
}
\caption{Elementwise variable $\epsilon$ penalty (EP) method for Navier Stokes}
\end{algorithm}

\section{Numerical Tests}\footnote{The datasets generated during and/or analysed during the current study are not publicly available but are available from the corresponding author on reasonable request.} In the numerical tests 5.1 and 5.2, the problems are tested using both elementwise penalty algorithm (Algorithm 1) and also this following coupled system: find $u^h\in X^h, p^h\in Q^h$ such that
\begin{equation}\label{eq:coupled}
\begin{split}
    \nu(\nabla u^h,\nabla v^h)-(p^h,\nabla\cdot v^h)=(f,v^h)\quad \forall v^h\in X^h,\\
    (\nabla\cdot u^h,q^h)=0\quad \forall q^h\in Q^h.
\end{split}
\end{equation}
\subsection{Test 1 taken from Burman and Hansbo \cite{burman2006edge}}
This model problem is constructed to test the convergence rate. The analytic solution is given below
\begin{equation}\label{test1}
\begin{aligned}
    u(x,y)&=20xy^3, \quad
    v(x,y)&=5x^4-5y^4, \quad
    p(x,y)&=60x^2y-20y^3-5.
\end{aligned}
\end{equation}
on $\Omega=(0,1)\times(0,1)$.
Inserting \eqref{test1} into Stokes equations \eqref{stokes} with $Re=100$ recovers the body force $f$. 

In this test, take $\epsilon$ lower bound $Lower Eps=10^{-8}$, global tolerance $TOL=10^{-5}$ and $LocTol_\Dt=\frac{1}{2}\frac{TOL^2}{|\Omega|}|\Dt|\approx1.5625\times10^{-11}$ for the case of 40 mesh points on each side. From Table \eqref{tab:test1-divu} with 40 mesh points per side: $\|\nabla\cdot u^h\|^2=5.49293\times10^{-6}<TOL$, the global tolerance condition satisfied using elementwise penalty. However from Figure \eqref{fig:triangle-divu40}(b): $\max\|\nabla\cdot u^h\|_\Dt^2\approx 1.15\times10^{-5}|\Dt|\approx3.59\times10^{-9}>LocTol$, the local condition does not satisfy but is very close to the local tolerance.

\begin{figure}[h]
    %\centering
  \subfigure[Coupled Stokes problem, the scale is about $10^{-1}$]{
	\begin{minipage}[c][1\width]{
	   0.45\textwidth}
	   \centering
	   \includegraphics[width=1\textwidth]{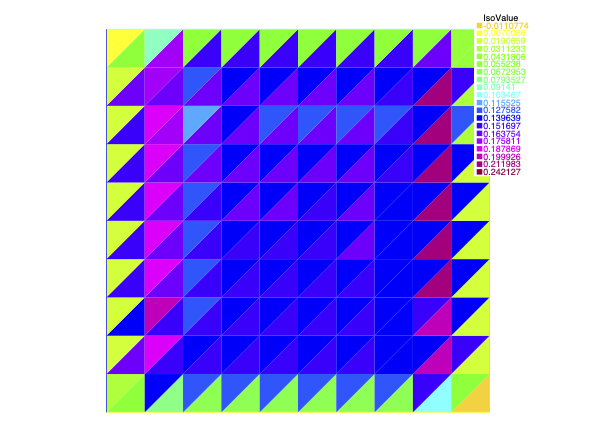}
	\end{minipage}
	}
 \hfill 	
  \subfigure[Elementwise penalty method (Algorithm1) for Stokes problem, the scale is about $10^{-3}$]{
	\begin{minipage}[c][1\width]{
	   0.45\textwidth}
	   \centering
	   \includegraphics[width=1\textwidth]{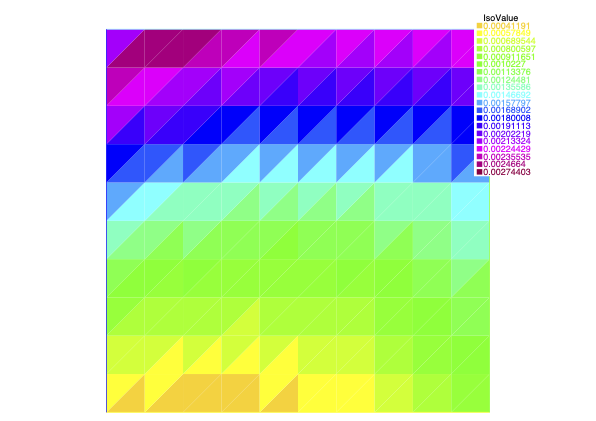}
	\end{minipage}
	}
\caption{$|\nabla\cdot u^h|_\Dt^2/|\Dt|$ with 10 mesh points on each side}
\label{fig:triangle-divu10}
\end{figure}

\begin{figure}[h]
    %\centering
  \subfigure[Coupled Stokes problem, the scale is about $10^{-3}$]{
	\begin{minipage}[c][1\width]{
	   0.45\textwidth}
	   \centering
	   \includegraphics[width=1\textwidth]{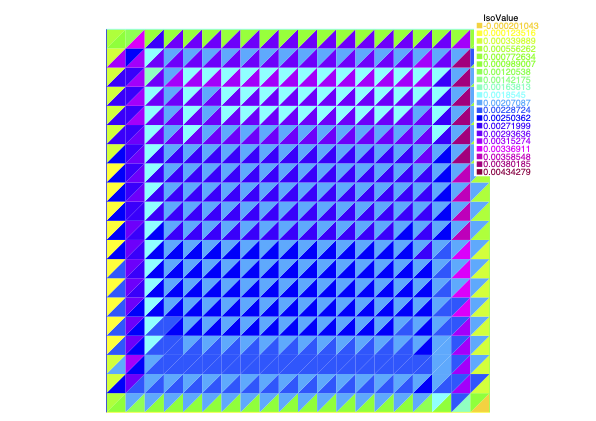}
	\end{minipage}
	}
 \hfill 	
  \subfigure[Elementwise penalty method (Algorithm1) for Stokes problem, the scale is about $10^{-4}$]{
	\begin{minipage}[c][1\width]{
	   0.45\textwidth}
	   \centering
	   \includegraphics[width=1\textwidth]{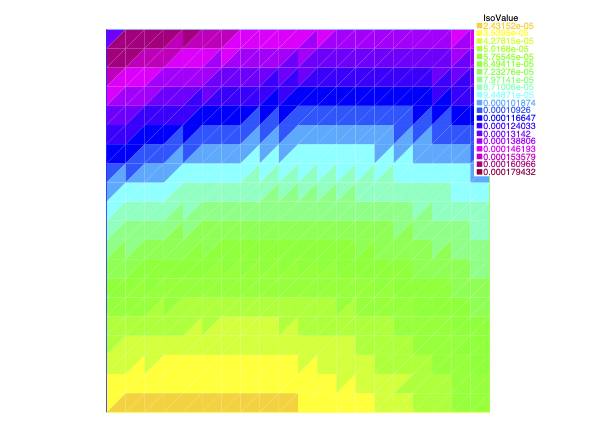}
	\end{minipage}
	}
\caption{$|\nabla\cdot u^h|_\Dt^2/|\Dt|$ with 20 mesh points on each side}
\label{fig:trianglle-divu20}
\end{figure}

\begin{figure}[h]
    %\centering
  \subfigure[Coupled Stokes problem, the scale is about $10^{-5}$]{
	\begin{minipage}[c][1\width]{
	   0.45\textwidth}
	   \centering
	   \includegraphics[width=1\textwidth]{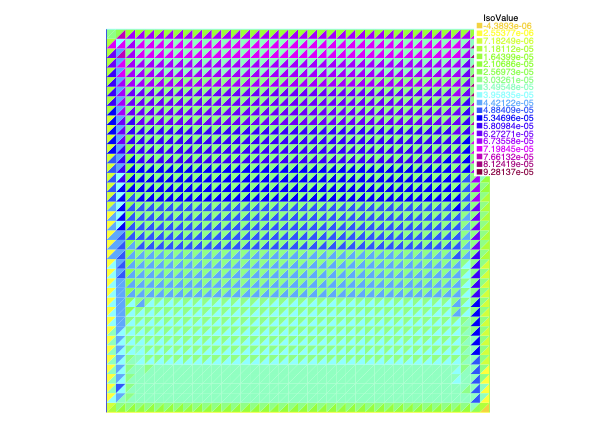}
	\end{minipage}
	}
 \hfill 	
  \subfigure[Elementwise penalty method (Algorithm1) for Stokes problem, the scale is about $10^{-6}$]{
	\begin{minipage}[c][1\width]{
	   0.45\textwidth}
	   \centering
	   \includegraphics[width=1\textwidth]{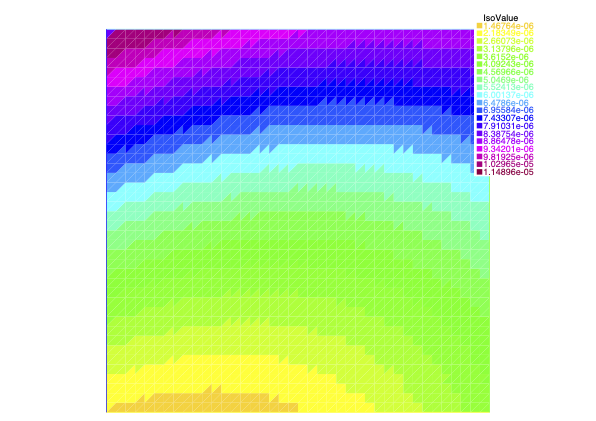}
	\end{minipage}
	}
\caption{$|\nabla\cdot u^h|_\Dt^2/|\Dt|$ with 40 mesh points on each side}
\label{fig:triangle-divu40}
\end{figure}
%++++++++++++

%++++++++++++++++++++++++++++
\begin{table}[H]
    \centering
    \begin{tabular}{||c|c|c|c||}
    \hline
     \# mesh points on each side    & coupled $\|u-u^h\|_{L^2}$ & penalty $\|u-u^h\|_{L^2}$ &rate \\
     \hline
    10     &0.00520688&0.00528456&- \\
    \hline
    20&0.000327941&0.00132306&1.99790\\
    \hline
    40&2.05561e-05&0.000340571& 1.95785
\\
    \hline
    \end{tabular}
    \caption{numerical error $\|u-u^h\|_{L^2}$ and convergence rate of elementwise penalty (compared with coupled system \eqref{eq:coupled})}
    %\label{tab:my_label}
\end{table}
\begin{table}[H]
    \centering
    \begin{tabular}{||c|c|c|c||}
    \hline
     \# mesh points on each side    & coupled $\|\nabla(u-u^h)\|_{L^2}$ & penalty $\|\nabla(u-u^h)\|_{L^2}$ &rate \\
     \hline
    10     &0.384253&0.433158&- \\
    \hline
    20&0.0494622&0.21608&1.00333\\
    \hline
    40&0.0062691&0.107975& 1.00087\\
    \hline
    \end{tabular}
    \caption{numerical error $\|\nabla(u-u^h)\|_{L^2}$ and convergence rate of elementwise penalty (compared with coupled system \eqref{eq:coupled})}
    %\label{tab:my_label}
\end{table}
\begin{table}[H]
    \centering
    \begin{tabular}{||c|c|c|c||}
    \hline
     \# mesh points on each side    & coupled $\|\nabla\cdot(u-u^h)\|_{L^4}^2$ & penalty $\|\nabla\cdot(u-u^h)\|_{L^4}^2$ &rate \\
     \hline
    10     &0.186365&0.00049467&- \\
    \hline
    20&0.00302458&3.12998e-05&3.98224\\
    \hline
    40&4.81016e-05&1.96239e-06& 3.99547\\
    \hline
    \end{tabular}
    \caption{numerical error $\|\nabla\cdot(u-u^h)\|_{L^4}^2$ and convergence rate of elementwise penalty (compared with coupled system \eqref{eq:coupled})}
    %\label{tab:my_label}
\end{table}
Table (5.1)-(5.3) presents the numerical error of Test1 of comparison between coupled system \eqref{eq:coupled} and elementwise penalty method (Algorithm 1). The convergence rate of the elementwise penalty are also presented in the fourth column.
\begin{table}[h]
    \centering
    \begin{tabular}{||c|c|c||}
    \hline
       \# mesh points on each side  & coupled $\|\nabla\cdot u^h\|^2$ & penalty $\|\nabla\cdot u^h\|^2$ \\
       \hline
    10     & 0.135344&0.00140525 \\
    \hline
    20 &0.002331& 8.78752e-05\\
    \hline
    40 &4.23739e-05& 5.49293e-06\\
    \hline
    \end{tabular}
    \caption{$\|\nabla\cdot u^h\|^2$ numerical result of  Test1}
    \label{tab:test1-divu}
\end{table}

\subsection{Test 2 Flow between offset cylinders taken from Layton and McLaughlin \cite{layton2019doublyadaptive}} This test is to test Algorithm 1 on a more complex flow problem and also a comparison between the coupled system and elementwise penalty scheme.

The domain is a disk with a smaller off-center disk inside. Let $r_1=1, r_2=0.1, c_1=0.5$ and $c_2=0$, the domain is given by
\begin{equation*}
    \Omega=\{(x,y):x^2+y^2\leq r_1^2\ \text{and}\ (x-c_1)^2+(y-c_2)^2\geq r_2^2\}.
\end{equation*}
We take Re=100 and the body force is given by
\begin{equation*}
    f(x,y)=(-4y(1-x^2-y^2),4x(1-x^2-y^2)).
\end{equation*}
In this test, $\epsilon$ lower bound $Lower Eps=10^{-10}$ and global tolerance $TOL=10^{-6}$. There are 60 mesh points on the outer circle and 30 mesh points on the inner circle. The mesh is denser near the inner circle. And for this mesh the shortest edge of all triangles is $min_e h_e=0.0220132 $ and the longest edge $max_e h_e=0.141732$. The smallest area of element triangle $min_\Dt |\Dt|=0.000166354$ and the largest area of triangle $max_\Dt |\Dt|=0.00528893$. The local tolerance $LocTol_\Dt=\frac{1}{2}\frac{TOL^2}{|\Omega|}|\Dt|$ ranges from $10^{-16}$ to $10^{-17}$. 

In this test, from Table \eqref{tab:test2-divu}: $\|\nabla\cdot u^h\|^2=1.01872\times10^{-19}<TOL^2$ and from Figure \eqref{fig:test2-triangle-divu}(b): $\max\|\nabla\cdot u^h\|_\Dt^2\approx8.59\times10^{-17}|\Dt|\approx 10^{-20}<LocTol_\Dt$. Here local condition and global condition are both satisfied. 
\begin{table}[h]
    \centering
    \begin{tabular}{||c|c||}
    \hline
        method & $\|\nabla\cdot u^h\|^2$  \\
        \hline
        coupled & 0.255675\\
        \hline
        elementwise penalty & 1.01872e-19\\
        \hline
    \end{tabular}
    \caption{numerical result $\|\nabla\cdot u^h\|^2$ of Test2 Stokes problem}
    \label{tab:test2-divu}
\end{table}
In the test using elementwise penalty (Algorithm 1) at final iteration,  $\epsilon_{max}=2.92232*10^{-8}$ and $\epsilon_{min}=10^{-10}$. 
\begin{figure}[h]
    %\centering
  \subfigure[Coupled Stokes problem]{
	\begin{minipage}[c][1\width]{
	   0.45\textwidth}
	   \centering
	   \includegraphics[width=1\textwidth]{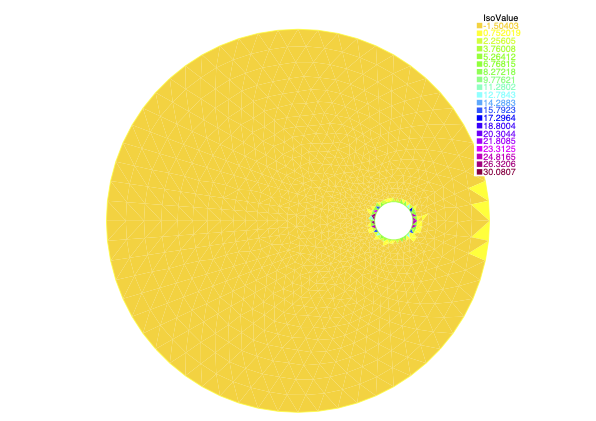}
	\end{minipage}
	}
 \hfill 	
  \subfigure[Elementwise penalty method (Algorithm1) for Stokes problem]{
	\begin{minipage}[c][1\width]{
	   0.45\textwidth}
	   \centering
	   \includegraphics[width=1\textwidth]{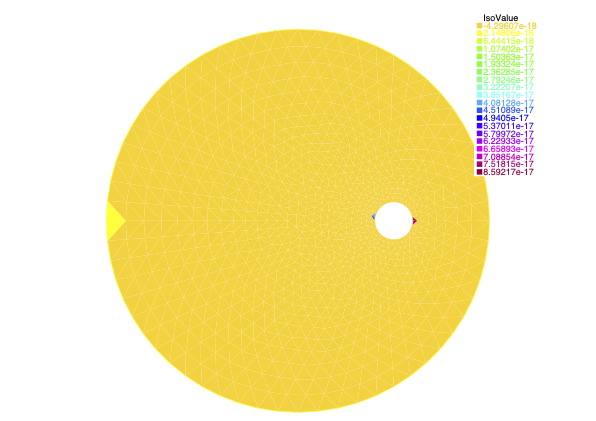}
	\end{minipage}
	}
\caption{$\|\nabla\cdot u^h\|_\Dt^2/|\Dt|$ of Test2, comparison between coupled \eqref{eq:coupled} and elementwise penalty system (Algorithm 1) (Note the scale in two plots are different. Coupled Stokes problem $\max_\Dt\|\nabla\cdot u^h\|_\Dt^2=O(10^2)$, elementwise penalty method $\max_\Dt \|\nabla\cdot u^h\|_\Dt^2=O(10^{-17})$)}
\label{fig:test2-triangle-divu}
\end{figure}
From Figure \eqref{fig:test2-triangle-divu}, the incompressibility condition is satisfied for the penalty method. For the coupled system $\max \|\nabla\cdot u^h\|_\Dt^2/|\Dt|\approx30.08$ which does not satisfy the incompressibility condition.
\begin{figure}[h]
    %\centering
  \subfigure[Coupled Stokes problem]{
	\begin{minipage}[c][1\width]{
	   0.45\textwidth}
	   \centering
	   \includegraphics[width=1\textwidth]{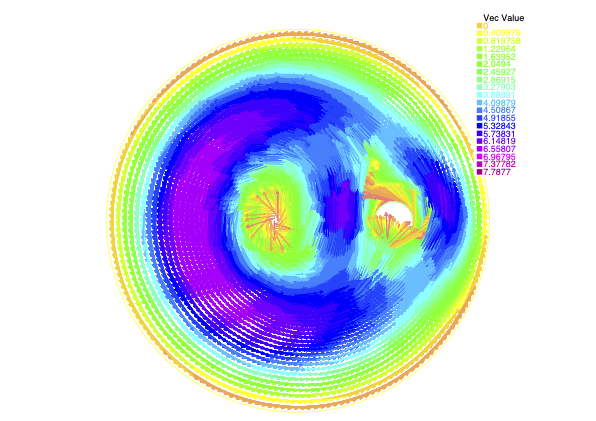}
	\end{minipage}
	}
 \hfill 	
  \subfigure[Elementwise penalty method (Algorithm1) for Stokes problem]{
	\begin{minipage}[c][1\width]{
	   0.45\textwidth}
	   \centering
	   \includegraphics[width=1\textwidth]{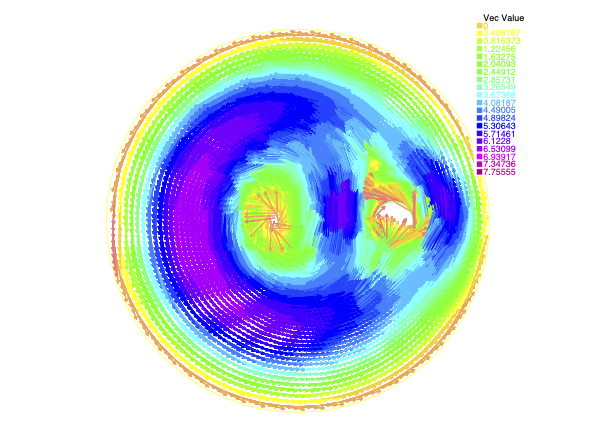}
	\end{minipage}
	}
\caption{velocity plot of Test 2, comparison between coupled \eqref{eq:coupled} and elementwise penalty system (Algorithm 1)}
\label{fig:test2-velocity}
\end{figure}
From the velocity plot Figure \eqref{fig:test2-velocity}, the coupled system and elementwise penalty system have similar results. But the elementwise penalty method has far smaller $\|\nabla\cdot u^h\|^2$ values. 

\subsection{Test3. Comparison test between constant penalty and elementwise penalty see Layton and Xu \cite{penalty-condition}}
In this test, we verify the adaptive elementwise penalty method (Algorithm 1) does better than normal constant penalty method by comparison Algorithm 1 with constant $\epsilon=10^{-8}\nu$ for all elements. Here constant $\epsilon=10^{-8}\nu$ is usually the approach used by engineering papers.

This comparison test problem is solved by using $P1$, conforming linear elements. Let the body force,
\begin{equation*}
    f(x,y) = (\sin(x+y), \cos(x+y))^T,
\end{equation*}
on $\Omega = (0,1)\times (0,1)$. In this test, $Re =1$, global tolerance $TOL = 10^{-6}$ and there are 40 mesh points on each side. The test results are shown in \Cref{tab:compare-constant-elementwise}.
\begin{table}[H]
    \centering
    \begin{tabular}{||c|c|c|||}
    \hline
         & constant penalty $\epsilon=10^{-8}$ & elementwise penalty (Algorithm 1) \\
     \hline
    $\|\nabla\cdot u^h\|^2$    &7.20178e-17&3.7741e-19 \\
    \hline
    average $\epsilon$&1e-8&0.000629366\\
    \hline
    \end{tabular}
    \caption{comparison of $\|\nabla\cdot u^h\|^2$ and average value of $\epsilon$ between constant penalty and elementwise penalty (Algorithm 1)}
    \label{tab:compare-constant-elementwise}
\end{table}
From \Cref{tab:compare-constant-elementwise}, constant penalty $\epsilon=10^{-8}$ is a ill conditioned linear system while elementwise penalty with average $\epsilon=6.3\times10^{-4}$ leads to a much better condioned system. And $\|\nabla\cdot u^h\|^2$ of adaptive elementwise penalty is smaller than constant penalty, thus adaptive elementwise penalty controls $\|\nabla\cdot u\|$ better than constant penalty method.

\subsection{Test4. Flow around a cylinder see Ingram \cite{Ross2013}, John, Matthies and Rang \cite{John2004}} This section is an extension of the elementwise penalty method test on the nonlinear Navier-Stokes equation (Algorithm 2). %We want to show derived from the Stokes problem; this scheme can also work on the time-dependent and nonlinear problem, e.g., Navier Stokes. 
Even though the local condition is only partially satisfied in this test, the global condition is satisfied and well controlled. 

The domain $\Omega$ is a $[0,2.2]\times[0,0.41]$ rectangle. The cylinder $S$ centered at $(0.2,0.2)$ with the diameter $0.1$ units.
The external force $f=0$, the final time is $T=8$ and the prescribed viscosity $\nu=10^{-3}$.
The flow has boundary conditions:
\begin{align*}
u(x,0,t)=u(x,0.41,t)=u|_{\partial\Omega_S}=(0,0)^T,\ \ &0\leq x\leq 2.2, \\
u(0,y,t) = u(2.2,y,t)=0.41^{-2}\sin(\pi t/8)(6y(0.41-y),0)^T,\ \ &0\leq y\leq 0.41.
\end{align*}
The mean inflow velocity is $U(t)=\sin(\pi t/8)$ such that $U_{max}=1$.\\
Let the initial condition satisfy the steady Stokes problem.
The following results using P3 finite element space for velocity. The number of degrees of freedom of velocity is 5091. The mesh is denser near cylinder S, and for this mesh, the shortest edge of all triangles is $min_e h_e=0.0101291 $ and the longest edge $max_e h_e=0.154404$. The smallest area of element triangle $min_\Dt |\Dt|=3.46846\times10^{-5}$ and the largest area of triangle $max_\Dt |\Dt|=0.00773693$. In this test, $\epsilon$ lower bound $LowerEps=10^{-10}$ and global tolerance $TOL=10^{-5}$. The local tolerance $LocTol_\Dt=\frac{1}{2}\frac{TOL^2}{|\Omega|}|\Dt|$ ranges from $10^{-13}$ to $10^{-15}$.
\begin{figure}[h]
    \centering
    \includegraphics[width=\textwidth]{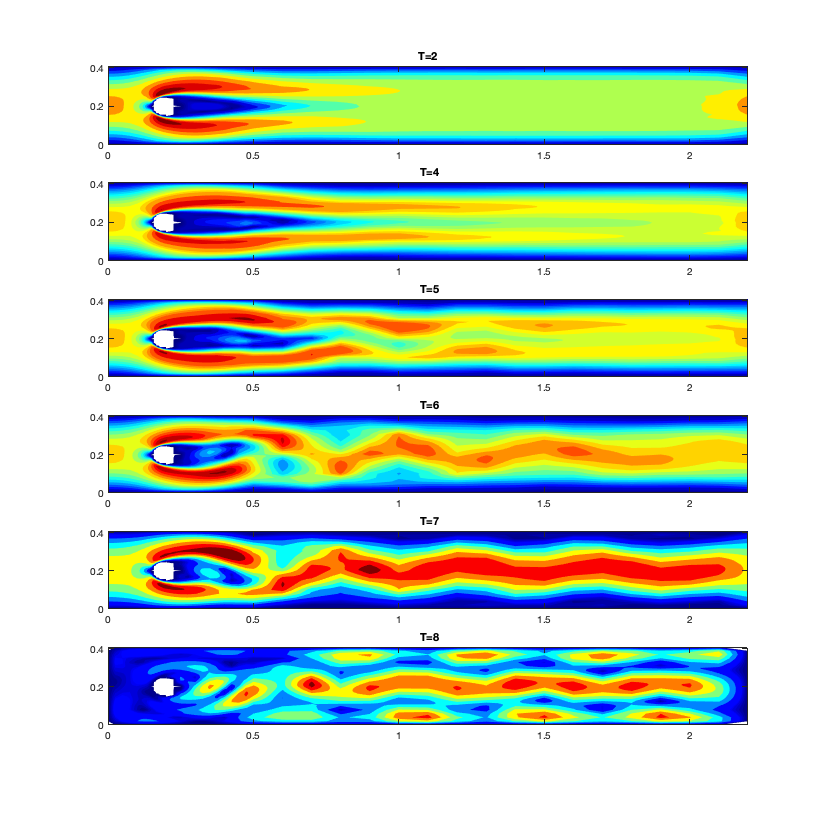}
    \caption{magnitude of velocity field at $T=2,4,5,6,7,8$ of Test 4 Algorithm 2 for NSE, $\Dt t=0.005$}
    \label{fig:test3-speed}
\end{figure}
Figure \eqref{fig:test3-speed} is the speed-profile at $T=2,4,5,6,7,8$ for flow with Re=1000. We can see the vortex shedding off the back of the cylinder in the test result.
\begin{figure}[h]
    \centering
    \includegraphics[width=0.9\textwidth]{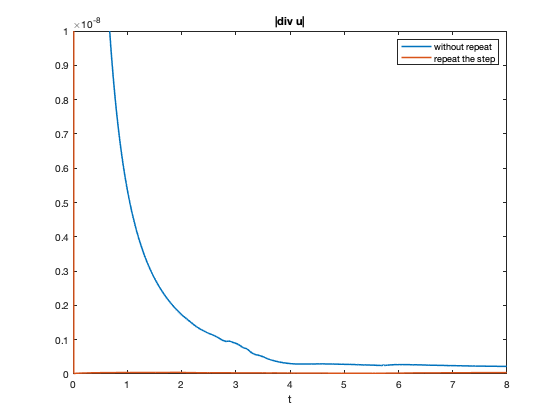}
    \caption{Plot of $\|\nabla\cdot u^h\|^2$ from T=0 to T=8}
    \label{fig:test3_divu_overtime}
\end{figure}
Figure \eqref{fig:test3_divu_overtime} is the plot of $\|\nabla\cdot u^h\|^2$ throughout the whole time interval. %and elementwise value $|\nabla\cdot u^h|^2_\Dt/|\Dt|$ at final time-step. Both global and local $|\nabla\cdot u^h|$ value are well controlled. 
The red curve (Algorithm 2 with step repeated) has smaller $\|\nabla\cdot u^h\|^2$ values than the blue curve (without repeating the step). Both global $\|\nabla\cdot u\|$ values are well controlled.%In this test, from Figure \eqref{fig:test3-divu}(a): throughout the time interval $[0,8]$: $\max_{t\in[0,8]}\|\nabla\cdot u^h\|^2=10^{-7}$. Here the global condition is satisfied throughout the whole time interval. 
\begin{figure}[h]
    %\centering
  \subfigure[$|\nabla\cdot u^h|_\Dt^2/|\Dt|$ at $T_{final}=8$ without repeating the step (Algorithm 2), the scale is about $10^{-8} \sim 10^{-9}$]{
	\begin{minipage}[c][1\width]{
	   0.45\textwidth}
	   \centering
	   \includegraphics[width=1\textwidth]{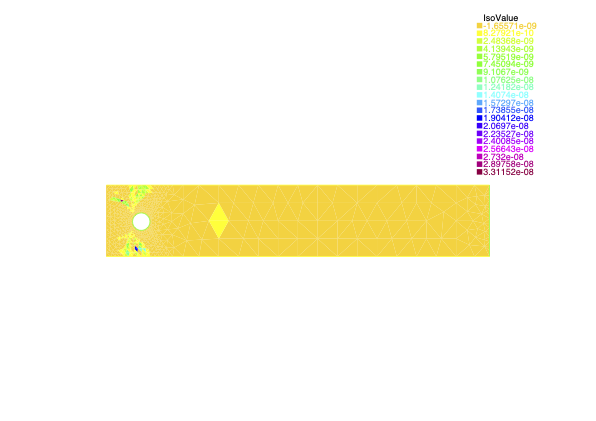}
	\end{minipage}
	}
 \hfill 	
  \subfigure[$|\nabla\cdot u^h|_\Dt^2/|\Dt|$ at $T_{final}=8$ with step repeated (Algorithm 2 with retry), the scale is about $10^{-11}\sim 10^{-12}$]{
	\begin{minipage}[c][1\width]{
	   0.5\textwidth}
	   \centering
	   \includegraphics[width=1\textwidth]{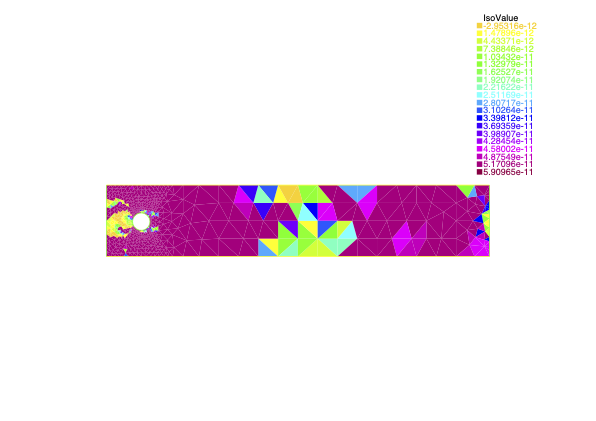}
	\end{minipage}
	}
\caption{result of Test 4 Algorithm 2 for NSE, $\Dt t=0.005$}
\label{fig:test3-divu}
\end{figure}
In order to check the local condition, we look at the elementwise value $|\nabla\cdot u^h|^2_\Dt/|\Dt|$ at the final time T=8. From Figure \eqref{fig:test3-divu}(a) without repeating the step: $\max\|\nabla\cdot u^h\|_\Dt^2\approx3\times10^{-8}|\Dt|\approx 10^{-11}$ slightly larger than the local tolerance $LocTol_\Dt$. From \eqref{fig:test3-divu}(b) with step repeated: $\max\|\nabla\cdot u^h\|_\Dt^2\approx 5\times10^{-11}|\Dt|\approx 10^{-14}$ satisfies the local tolerance. For Algorithm 2 with step repeated, the global and local $\|\nabla\cdot u^h\|$ values are smaller but need more computing time compared with Algorithm 2 without retry. For Algorithm 2 without repeating the step, the overall result is satisfying even though the local conditions are only partially satisfied.

\section{Conclusions} In this paper, we proposed a new variable $\epsilon$ penalty method starting from the Stokes problem. We proved the stability and derived an error approximation of the new pointwise penalty (PP) \eqref{vform} on the Stokes problem. And at the end, we test the algorithm on the Stokes problem and extend it to test the time-dependent nonlinear Navier Stokes problem using elementwise penalty (EP) \eqref{vform2}. This is just a start of this new scheme, there are plenty of improvements possible. Picking the right global tolerance TOL and maximum iteration MaxIter is still a problem to consider. Algorithm 2 is new, we currently do not know if or not we need to repeat each time-step after setting the new $\epsilon$. %We may consider a different type of stopping criterion in the future. 
We emphasize that our target is the 3d, time-dependent NSE problem for which the method is implemented as Algorithm 2, 
%(and in greater detail in {\cc cite paper here}) 
without appreciable complexity increase over simple, linear constant $\epsilon$ penalty methods.

In this paper, we focused on the velocity and did not pay attention to the accuracy of pressure. Pressure recovery is also a big problem to consider. In Kean and Schneier \cite{Kean2019error}, two different pressure recovery methods are introduced and analyzed. As for the time-dependent problem, only constant time-step schemes are considered in this paper. To further optimize the algorithm, adding a time filter Guzel and Layton \cite{Guzel2017tilter,guzel2018timefilter} and adapt the time-step is also a good research direction in the future.
Both the stability and error analysis is given based on the assumption that the grad-div term can be replaced by the variational form \eqref{grad-div-vform1}. The numerical analysis based on assumption \eqref{grad-div-vform2} (i.e. elementwise penalty) is also an interesting problem.

\section*{Acknowledgement}
I would like to thank Professor William Layton for his brilliant idea for constructing the model and his guidance during the research.

\bibliographystyle{plain}
\bibliography{main}
%\printbibliography
\end{document}